\documentclass[11]{siamltex}

\usepackage{datetime}
\settimeformat{ampmtime}
\usdate

\newcommand{\nwc}{\newcommand}
\nwc{\draftdate}{\today}
\usepackage{showkeys}
\usepackage[dvipsnames]{xcolor}
\usepackage[top=1in, bottom=1in, left=1in, right=1in]{geometry}
\usepackage{subcaption}
\usepackage{xcolor}
\usepackage[pdftex]{graphicx}
\usepackage{psfrag}
\usepackage{wrapfig}
\usepackage{epsf}
\usepackage{amsmath,amssymb}
\usepackage{scalerel,amssymb}
\allowdisplaybreaks
\usepackage{url}
\usepackage{paralist}
\usepackage{comment}
\usepackage{chngcntr}
\usepackage{booktabs}
\counterwithout{figure}{section}
\counterwithout{table}{section}
\renewcommand{\theenumi}{(\roman{enumi})}

\newcommand{\new}[1]{#1}

\newcommand{\patch}[1]{#1}
\newcommand{\redbegin}{\begingroup}
\newcommand{\redend}{\endgroup}

\makeatletter
\@ifundefined{proposition}{\newtheorem{proposition}[theorem]{Proposition}}{}
\makeatother

\newcommand{\barint}{\hbox{$\int$\kern-0.75\intwidth
\vrule width 0.5\intwidth height 2.4pt depth -2pt\kern0.25\intwidth}}
\newlength\intwidth
\setbox0=\hbox{$\int$}
\intwidth=\wd0

\newcommand\avint{\hbox{\hbox{$\displaystyle \int$}\hbox{\kern-.9em{$-$}}}}
\newcommand\smavint{\hbox{\hbox{$\int$}\hbox{\kern-.75em{$-$}}}}

\nwc{\st}{^{\mbox{\it st}}}
\nwc{\qref}[1]{(\ref{#1})}
\nwc{\veloc}{v}
\nwc{\rhoc}{\beta}
\nwc{\hl}{\hat{L}}

\def\Xint#1{\mathchoice
{\XXint\displaystyle\textstyle{#1}}%
{\XXint\textstyle\scriptstyle{#1}}%
{\XXint\scriptstyle\scriptscriptstyle{#1}}%
{\XXint\scriptscriptstyle\scriptscriptstyle{#1}}%
\!\int}
\def\XXint#1#2#3{{\setbox0=\hbox{$#1{#2#3}{\int}$}
\vcenter{\hbox{$#2#3$}}\kern-.51\wd0}}

\def\dashint{\Xint-}

\nwc{\intRp}{\int_0^\infty}
\nwc{\aint}{\dashint}
\nwc{\aaint}{\dashint}

\newcommand{\tE}{{\tilde E}}
\newcommand{\tW}{{\tilde W}}
\newcommand{\tWone}{{\tilde W_1}}

\newcommand{\R}{\mathbb R}

\newcommand{\be}{\begin{eqnarray}}
\newcommand{\ee}{\end{eqnarray}}
\newcommand{\nn}{\nonumber}
\newcommand{\ben}{\begin{eqnarray*}}
\newcommand{\een}{\end{eqnarray*}}

\begin{document}

\title{
Minimizers for the Cahn--Hilliard Energy Functional with the {Flory--Huggins} Potential
under Strong Anchoring Conditions
}

\author{
Shibin Dai
\thanks{
Department of Mathematics, University of Alabama, Tuscaloosa, AL 35487-0350, USA.
 Email: sdai4@ua.edu.
}
\and
Abba Ramadan
\thanks{
Department of Mathematics, University of Alabama, Tuscaloosa, AL 35487-0350, USA.
Email: aramadan@ua.edu
}
\and
Natasha Sharma
\thanks{Department of Mathematical Sciences, University of Texas El Paso, El Paso, TX 79968, USA.
Email: nssharma@utep.edu
}
}

\maketitle

\begin{abstract}
In this paper, we theoretically and numerically study the minimizers for the Cahn--Hilliard energy
with the {Flory--Huggins} potential under the strong anchoring condition{,} i.e., the
Dirichlet boundary condition. We reveal bifurcation phenomena mediated by the boundary condition,
the {transition layer thickness}, and the temperature of the system. Numerical simulations are
also presented to approximate the minimizers of this energy by solving a gradient-flow
equation{,} namely the Allen--Cahn equation constrained with strong anchoring conditions and
random initial data. The {effects of varying the transition layer thickness and temperature} are
presented to
\patch{illustrate, and are consistent with, the theoretical analysis.}
\end{abstract}


\section{Introduction}
\label{s:Introduction}

The Cahn--Hilliard functional is a classical model for the free energy of a binary mixture, and it
has been
\patch{widely used in many applications.}
In particular, it is the cornerstone for the
diffuse-interface model to characterize the free energy of a system undergoing phase separation
\cite{CH1, CH2}. For a system occupying a bounded \patch{connected} domain $\Omega\subset\R^n$ with a $C^1$ boundary,
the functional can be written as
\begin{align}\label{def-E}
	E(u) = \int_\Omega \left(  \frac{\kappa}{2} |\nabla u|^2 + W(u)\right)dx.
\end{align}
Here, $W$ is a double-well potential, with two minimizers corresponding to the two phases of the
system, and $\kappa$ is a parameter measuring the thickness of the transition layer between the two
phases.

The Cahn--Hilliard functional, along with the Cahn--Hilliard and Allen--Cahn equations, provides a
central framework for modeling key phenomena in two-phase materials, including phase separation,
coarsening, and pattern formation. Since such processes are studied on a domain $\Omega$, the role
of the boundary $\partial\Omega$ is equally significant. A commonly imposed condition is the Neumann
boundary condition $\partial_n u=0$, where $n$ is the outward unit normal to $\partial\Omega$
\cite{BCS,BN,BN1,CEC,EG,GNS,LM,P,WKG}. Other boundary conditions are also frequently considered,
such as periodic conditions, especially in computational studies
\cite{CR,CPW,CS,DD,DD1,FYHLDC,GCBH,H,HLT,Y}.

Dai, Li, and Luong \cite{DLL} recently studied minimizers of the Cahn--Hilliard functional with
quartic double-well potential under Dirichlet (strong anchoring) conditions, where $u$ is fixed to
a prescribed boundary function $g$ pointwise on $\partial\Omega$. This represents the most rigid
anchoring and is of considerable physical relevance \cite{DN,LO}. A weaker anchoring can instead be
formulated by requiring $u$ to approximate $g$ within a tolerance in a suitable norm such as $L^2$.
Du and Nicolaides \cite{DN}
\patch{considered the Dirichlet boundary condition in a finite-element study of the one-dimensional Cahn--Hilliard equation,}
while Bronsard and Hilhorst \cite{BH1} analyzed its asymptotic effects via energy methods. Further
work in this direction can be found in \cite{BH2,GL,LJSK}.

Although a quartic potential is theoretically easier and numerically convenient, a physically more
realistic choice is the {Flory--Huggins} potential \cite{A,AW,CMZ,DD2,M,MZ}, also commonly
referred to as the logarithmic potential,
\begin{align} \label{def-W}
	W(u) = \left\{
	\begin{array}{ll}
	\dfrac{\theta}{2}\left( (1-u)\ln (1-u) + (1+u)\ln (1+u) \right) + \dfrac{1}{2} (1-u^2)
	& \mbox{if } |u|\leq 1, \\[6pt]
	+\infty & \mbox{otherwise.}
	\end{array}
	\right.
\end{align}
\new{Here we use the fact $\lim_{s\to 0^+} s\ln (s) = 0$ for the definition of $W(u)$ when $u\to 1^-$ and $u\to (-1)^+$. Hence 
$W(u)$ has jump discontinuities, or infinite barriers, at $u=\pm 1$.}  
\patch{The derivative $W'(u)=\frac{\theta}{2}\ln\frac{1+u}{1-u}-u$ has logarithmic singularities at $u=\pm 1$. The strict bound $|u|\le u_\theta<1$ for minimizers is not part of the definition of $W$; it is proved in Lemma~\ref{lemma-E} (and refined to $u_+^*<u_\theta$ in Lemma~\ref{lemma-upper}).}
The logarithmic terms describe the entropy of mixing of the binary system \cite{MZ}, and the
\new{infinite barriers} of $W$ at $u = \pm 1$ enforces the physical constraint \new{$|u|\leq1$}, ensuring that the order
parameter remains within the physically admissible range of concentrations. In contrast, the quartic
approximation is reasonable only when the quench is shallow, i.e., when the temperature $\theta$ is
close to the critical temperature \cite{MZ}. Here $\theta>0$ is the rescaled temperature of the
system. If $\theta\geq 1${,} then $W$ is strictly convex and has only one minimizer at $u=0${,}
and hence $E$ is a convex functional. If $0<\theta <1${,} then $u=0$ is a local maximizer for
$W$ and there are two minimizers $\pm u_\theta${,} where $u_\theta$ satisfies $0<u_\theta<1$
and $W(u_\theta)= W(-u_\theta)$, {see Fig. \ref{fig-log-potential}. It is easy to prove that $u_\theta\to 1^-$ as $\theta\to 0^+$.}  
We only consider $0<\theta<1$.  Also, see \cite{BB,CEC} for works
that focused on homogeneous Neumann or periodic boundary conditions, and for the numerical scheme
with dynamic boundary conditions see \cite{GWWZ}.
\begin{figure}
\centering 
\includegraphics[height=1in]{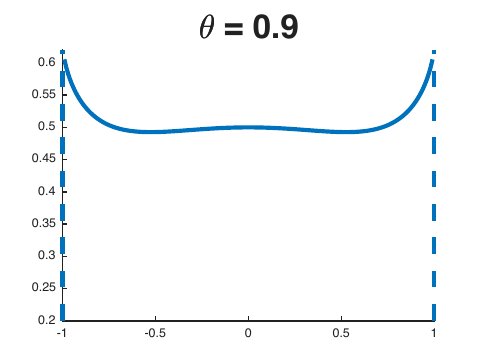}
\includegraphics[height=1in]{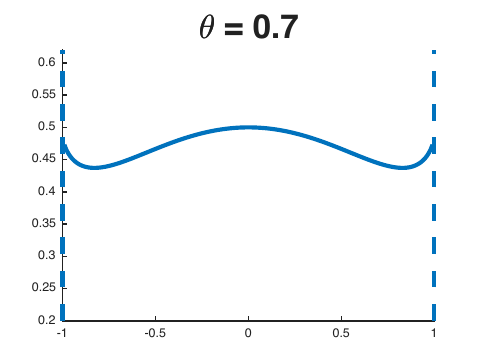}
\includegraphics[height=1in]{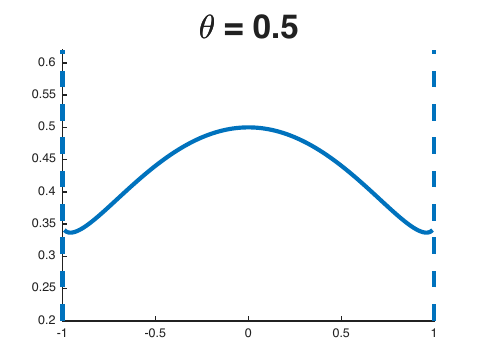}
\includegraphics[height=1in]{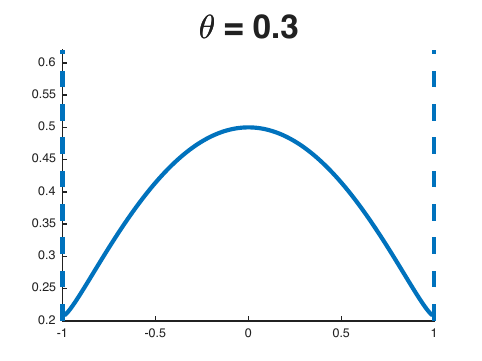}
\caption{The Flory-Huggins logarithmic potential with different values of $\theta$.}
\label{fig-log-potential}
\end{figure}

In \cite{DLL}, Dai, Li, and Luong found that, for the quartic double-well potential, when the
boundary condition is a homogeneous mixture of the two phases, there is a bifurcation {phenomenon}
mediated by $\kappa$, the thickness of the transition layer. To be more precise, when $\kappa$ is
smaller than a critical value which is proportional to the inverse of the smallest eigenvalue of
the negative Laplace \new{operator with homogeneous Dirichlet boundary condition}, there will be two nontrivial 
symmetric minimizers for $E$. In contrast,
when the boundary value is larger (or smaller) than the homogeneous mixture, the symmetry of the
system is broken and there is only one minimizer. Their approach involves explicit calculations that
only work for the quartic potential, and it was left as an open problem {whether bifurcation
results of this kind hold for} double-well potentials of other forms.

In \cite{Dai-Ramadan:minimizers}, Dai and Ramadan studied the minimizers of the
{de~Gennes--Cahn--Hilliard} energy under strong anchoring conditions. They found that, after a
nonlinear transformation, the {de~Gennes--Cahn--Hilliard} energy can be transformed into a
Cahn--Hilliard energy with a potential that is a trigonometric function. {They} then took
advantage of the additional structure provided by the trigonometric function to derive bifurcation
results mediated by the transition layer thickness $\kappa$. It was still left open {whether
such results extend to} double-well potentials of other forms, or whether {a} unified approach
to generic double-well potentials {is possible}.

{Our goal in this paper is to resolve this open problem for the physically relevant
Flory--Huggins potential. Specifically, we study} the minimizers of the Cahn--Hilliard energy $E$
\qref{def-E} with the {Flory--Huggins} potential \qref{def-W} for $0<\theta<1$, under the
homogeneous strong anchoring condition, i.e., the Dirichlet boundary condition $u=0$ on
$\partial\Omega$. {In contrast to the quartic and trigonometric settings,
\patch{the infinite barriers of the Flory--Huggins potential at $u=\pm1$, and the resulting logarithmic singularity of $W'$, require substantially different techniques.}
Moreover,
the rescaled temperature $\theta$ enters as an additional bifurcation parameter alongside the
transition layer thickness $\kappa$, giving rise to a richer bifurcation structure.} Our main
result is the following.

\begin{theorem} \label{th-main}
Let $\lambda_1$ be the smallest eigenvalue of the negative Laplace operator in $\Omega$ \new{with homogeneous 
Dirichlet boundary condition}.
\begin{enumerate}
	\item If $\kappa \geq (1-\theta)/\lambda_1$, then $E$ has a unique minimizer $u^*\equiv 0$ \new{in $H_0^1(\Omega$)}
	and $E(u^*) = \frac{|\Omega|}{2}$.
	\item If $0<\kappa < (1-\theta)/\lambda_1$, then $E$ has \new{exactly} two nontrivial minimizers $u^*_+$
	and $u^*_-$ \new{in $H_0^1(\Omega)$}, with $u^*_- = - u^*_+$, and $0<u^*_+<u_\theta$ in $\Omega$,
	and $E(u^*_{\new{+}}) \new{=E(u^*_{-}) } <\frac{|\Omega|}{2}$.
\end{enumerate}
\end{theorem}

The rest of the paper is structured as follows. Sections~2--5 are devoted to the proof of
Theorem~\ref{th-main}, each section handling a different aspect of the {problem}. In Section~2
we discuss approximations and modifications of the {Flory--Huggins} potential
\patch{and prove the existence of minimizers}.
\patch{In Section~3 we study minimizers for the Cahn--Hilliard energy with the {Flory--Huggins} potential using the Nehari manifold and the scaling functions (fibers) $s\mapsto E(su)$, and in Section~4 we study minimizers along fibers for the modified energy $\tE$.}
In Section~5 we study properties of minimizers.
In Section~6 we present numerical simulation results. Finally{,} in Section~7 we discuss
{challenges and} topics for further exploration.

\section{Approximations and modifications}
Having stated the main result, we now turn to its proof, which occupies Sections~2--5.
The key difficulty is the \new{infinite barriers} of the Flory--Huggins potential \new{at $u=\pm 1$}, which
prevents direct variational arguments. \new{However, the minimizers of $E$ over $H_0^1(\Omega)$ are bounded and away from 
the infinite barriers, as stated in the following lemma.}

\new{
\begin{lemma} \label{lemma-E}
	$E$ has at least one minimizer in $H_0^1(\Omega)$. In addition, all minimizers of $E$ in $H_0^1(\Omega)$ 
    satisfy $-u_\theta \leq u(x) \leq u_\theta$ a.e. { in} $\Omega$.
\end{lemma}
}

\begin{proof}
\new{We first show that $E$ has a minimi\patch{z}er $u$ in $H_0^1(\Omega)$ and that $-u_\theta \leq u\leq u_\theta$.   
Since $E(u) \geq 0$ for all $u\in H_0^1(\Omega)$ and $E(0) = |\Omega|/2 <\infty$, the infimum of $E$ over $H_0^1(\Omega)$ is finite. Define 
\[ m:=\inf\{E(u): u\in H_0^1(\Omega)\}.\] 
Let $\{u_k\in H_0^1(\Omega)\}_{k=1}^\infty$ be a minimizing sequence satisfying $E(u_k)\to m$ as $k\to\infty$. Then there is $M>0$
such that $E(u_k)\leq M$ for all $k$.
Let 
\begin{align} \label{def_hat-u}
\hat u_k(x) = \left\{
\begin{array}{ll}
    u_\theta & \mbox{ if } u_k(x) \geq u_\theta, \\
    u_k(x) & \mbox{ if } -u_\theta < u_k(x) < u_\theta, \\
    - u_\theta & \mbox{ if } u_k(x) \leq  - u_\theta.
\end{array}
\right. 
\end{align} 
We note that $|\hat u_k(x)| \leq u_\theta$ for a.e. $x\in \Omega$ and all $k$. 
In addition, since $W(u_\theta) = W(-u_\theta)\leq W(s)$ for all $s\in \R$, we have $W(\hat u_k) \leq W(u_k)$ for all $k$.}

\new{The weak derivative of $\hat u_k$ is   
\begin{align}
\nabla \hat u_k(x) = \left\{
\begin{array}{ll}
    0 & \mbox{ if } u_k(x) \geq u_\theta, \\
    \nabla u_k(x) & \mbox{ if } -u_\theta < u_k(x) < u_\theta, \\
    0 & \mbox{ if } u_k(x) \leq  - u_\theta,
\end{array}
\right. 
\end{align}
for a.e. $x\in\Omega$.  
Therefore, $\int_\Omega |\nabla\hat u_k|^2\,dx \leq \int_\Omega |\nabla u_k|^2\,dx$ 
for all $k$. This gives 
\[ E(\hat u_k) \leq E(u_k) \mbox{ for all }k\]
and $\limsup_{k\to\infty} E(\hat u_k) \leq m$. 
Noticing $\hat u_k(x) = u_k(x) = 0$ for all $x\in\partial\Omega$, we have $\hat u_k\in H_0^1(\Omega)$
and hence $\liminf_{k\to\infty} E(\hat u_k) \geq m$, which gives $\lim_{k\to\infty} E(\hat u_k) =m$ and 
$\hat u_k$ is a minimizing sequence in $H_0^1(\Omega)$.}

\new{
Since $\int_\Omega |\nabla\hat u_k|^2\,dx \leq 2E(\hat u_k)/\kappa \leq 2M/\kappa$, $\hat u_k$ is 
bounded in $H_0^1(\Omega)$ by Poincar\'e's inequality. By the compact embedding theorem, there exists a subsequence 
$\{\hat u_{k_j}\}$ and $u\in H_0^1(\Omega)$ 
such that 
\begin{align*} 
\hat u_{k_j} &\rightharpoonup u \mbox{ weakly in } H_0^1(\Omega), \\
\hat u_{k_j} &\to u\mbox{ strongly in }L^2(\Omega) \mbox{ and a.e. in }\Omega.  
\end{align*}
By the weak lower semicontinuity of norms, we have 
\begin{align} \label{liminf}
 {\int_\Omega |\nabla u|^2\,dx \leq \liminf_{j\to\infty}\int_\Omega|\nabla\hat u_{k_j}|^2\,dx.}
\end{align} 
Since $|\hat u_{k_j}|\leq u_\theta$ for all $x\in\Omega$, we have $|u(x)|\leq u_\theta$ for a.e. $x\in \Omega$. 
Since $W$ is continuous on $[-u_\theta, u_\theta]$, $W(\hat u_{k_j}(x)) \to W(u(x))$ for a.e. $x\in\Omega$, and 
\[ |W(\hat u_{k_j}(x))|\leq \max\{|W(s)|:\; |s|\leq u_\theta\}, \quad |W(u(x))|\leq \max\{|W(s)|:\; |s|\leq u_\theta\} \]
for a.e. $x\in\Omega$. By the bounded convergence theorem, we have 
\patch{$\lim_{j\to\infty}\int_\Omega W(\hat u_{k_j})\,dx = \int_\Omega W(u)\,dx$.} Combined with \qref{liminf} we obtain
\begin{align}
   m\leq  E(u) \leq \liminf_{j\to\infty} E(\hat u_{k_j}) \leq \liminf_{j\to\infty} E(u_{k_j}) = m.
\end{align}
So $E(u) = m$ and $u$ is a minimizer for $E$ over $H_0^1(\Omega)$. 
}

\new{Now if $w$ is any minimi\patch{z}er for $E$ over $H_0^1(\Omega)$, we can define $\hat w$ through a cutoff as in \qref{def_hat-u}. 
The same argu\patch{m}ent as above gives $\int_\Omega |\nabla \hat w|^2\,dx \leq \int_\Omega |\nabla w|^2\,dx$ and $\int_\Omega W(\hat w)\,dx \leq \int_\Omega W(w)\,dx$ and hence
$E(\hat w)\leq E(w)$. If the set 
\[ \{ x\in\Omega: |w(x)| > u_\theta\}\]
has positive measure, then on this set $W(w)> W(u_\theta)$ and we have 
\begin{align}
    \int_\Omega W(w)\,dx - \int_\Omega W(\hat w)\,dx = \int_{\{ x\in\Omega: |w(x)| > u_\theta\}} (W(w) - W(u_\theta))\,dx >0,
\end{align}
which gives $\int_\Omega W(\hat w)\,dx < \int_\Omega W(w)\,dx$  and $E(\hat w) < E(w)$, which contradicts the assumption that $w$ is a minimizer. So  
$w$ must satisfy $-u_\theta\leq w\leq u_\theta$ a.e. in $\Omega$.}
\end{proof}

\new{ To avoid potential issues due to the infinite barriers at $u=\pm1$, in} this section we construct a \new{continuous and 
piecewise smooth} global
approximation $\tW$ of $W$ that coincides with $W$ on the physically relevant region
$|u|\leq u_\theta$ and shares the same minimizers. 
 To address this, we separate the logarithmic terms from the smooth term. Specifically, $W$
can be written as the difference of two convex functions,
\begin{align}{\label{W-decomp}}
	{W(u) = W_1(u) - W_2(u),}
\end{align}
{where}
\begin{align} \label{def-W1}
	W_1(u) = \left\{
	\begin{array}{ll}
	\dfrac{\theta}{2}\left( (1-u)\ln (1-u) + (1+u)\ln (1+u) \right) & \mbox{if } |u|\leq 1, \\[6pt]
	+\infty & \mbox{otherwise,}
	\end{array}
	\right.
\end{align}
and $W_2(u) = \frac{1}{2}(u^2 -1)$. {Note that $-W_2(u) = \frac{1}{2}(1-u^2)$, which is
consistent with the definition \qref{def-W}.} For $-1<u<1$, we have $W'(u) = W_1'(u) - u$ and
$W''(u) = W_1''(u) - 1${, where}
\begin{align}
	W_1'(u) & = \frac{\theta}{2} \left( - \ln (1-u) + \ln (1+u) \right)
	= \theta \left( u + \frac{1}{3} u^3 +\frac{1}{5} u^5+\cdots \right), \label{W1'-taylor}\\
	W_1''(u) &= \frac{\theta}{1-u^2}  = \theta(1+ u^2 + u^4 + u^6+\cdots). \label{W1''-taylor}
\end{align}

Since $W_1'(u) \to \infty$ and $W_1''(u)\to \infty$ as $u\to 1^-$, there exists
$\hat{u}_\theta \in (u_\theta, 1)$ such that $W_1'(u) > 100$ and $W_1''(u) >100$ for all
$\hat{u}_\theta \leq u< 1$. By the Taylor expansions \qref{W1'-taylor} and \qref{W1''-taylor},
there exists $k$ such that
\begin{align}
	 \theta \left( u + \frac{1}{3} u^3 +\frac{1}{5} u^5+\cdots + \frac{1}{2k+1}u^{2k+1} \right) > 100, \quad
	 \theta(1+ u^2 + u^4 + \cdots + u^{2k}) > 100 \nn
\end{align}
for all $u \geq \hat{u}_\theta$. Define a new function
\begin{align} \label{def-tW1}
	\tWone (u) :=\left\{
	\begin{array}{ll}
		 W_1(\hat{u}_\theta) + \displaystyle\int_{\hat{u}_\theta}^u
		 \theta \left( t + \frac{1}{3} t^3 + \cdots+ \frac{1}{2k+1}t^{2k+1}\right)dt
		 & \mbox{if } u>\hat{u}_\theta, \\[6pt]
		W_1(u) &  \mbox{if } |u|\leq \hat{u}_\theta, \\[6pt]
		W_1(-\hat{u}_\theta) + \displaystyle\int_{-\hat{u}_\theta}^u
		\theta \left( t + \frac{1}{3} t^3 + \cdots+ \frac{1}{2k+1}t^{2k+1}\right)dt
		& \mbox{if } u< -\hat{u}_\theta.
	\end{array}
	\right.
\end{align}
Then $\tWone$ is strictly convex in {each of the intervals} $(-\infty, -\hat{u}_\theta)$,
$(-\hat{u}_\theta, \hat{u}_\theta)$, and $(\hat{u}_\theta, \infty)${.} Define
\begin{align}{\label{def-tW}}
	{\tW(u) := \tWone(u) - W_2(u).}
\end{align}
Then $\tW(u)$ is defined for all $u\in\R$, {with} $\tW(u) = W(u)$ for all
$|u| \leq \hat{u}_\theta$ and $\tW(u) < W(u)$ for all $|u|> \hat{u}_\theta$. In addition, $\tW$
is strictly convex in $(-\infty, -\hat{u}_\theta)$ and $(\hat{u}_\theta, \infty)$, respectively.

\begin{lemma}
	$\tW(u)$ has {exactly} two minimizers $\pm u_\theta$, and
	$\tW(u_\theta) = \tW(-u_\theta)$.
\end{lemma}
\begin{proof}
	Since $\tW$ is an even function and $\tW(u) = W(u)$ for all $|u| \leq \hat{u}_\theta$, we
	only need to {show} that $\tW$ is strictly increasing {on} $[\hat{u}_\theta, +\infty)$.
	To {see this, observe} that
	\[ \tW'(\hat{u}_\theta) = \tWone'(\hat{u}_\theta) - \hat{u}_\theta
	= \theta\left(\hat{u}_\theta + \frac{1}{3} \hat{u}_\theta^3 + \cdots
	+ \frac{1}{2k+1} \hat{u}_\theta^{2k+1}\right) - \hat{u}_\theta > 100 - 1 = 99>0,\]
	and for $u > \hat{u}_\theta${,}
	\redbegin
	\begin{align}
		\tW''(u)  = \theta ( 1+ u^2 + u^4 + \cdots + u^{2k}) - 1
		> \theta ( 1+ \hat{u}_\theta^2 + \hat{u}_\theta^4 + \cdots + \hat{u}_\theta^{2k}) - 1
		> 100 - 1 = 99 > 0. \nn
	\end{align}
	\redend
	{Consequently, for all $u > \hat{u}_\theta$,}
	\[\tW'(u) > \tW'(\hat{u}_\theta) > 99 > 0,\]
	and hence $\tW$ is strictly increasing {on} $[\hat{u}_\theta, +\infty)$.
\end{proof}


Define the {modified} energy
functional
\begin{align} \label{def-tilde_E}
	\tE(u) := \int_\Omega \left( \frac{\kappa}{2} |\nabla u|^2 + \tW(u) \right) \,dx.
\end{align}
By the same cutoff argument \new{as in Lemma \ref{lemma-E}}, all minimizers of $\tE$ also satisfy
$-u_\theta \leq u(x) \leq u_\theta$ a.e.\ in $\Omega$. This means that $E$ and $\tE$ have
exactly the same minimizers over $H_0^1(\Omega)$. {Moreover,} $E(u) \geq \tE(u)$ for all
$u\in H_0^1(\Omega)$.
\redbegin
More precisely, $\tE$ attains its minimum. Indeed, by the preceding lemma $\tW (u)\ge\tW(u_\theta)=W(u_\theta)$ for all $u$, 
so truncation at $\pm u_\theta$ (see \qref{def_hat-u}) does not increase $\tW$, and $\tW=W$ on $\{|u|\le u_\theta\}$ because $u_\theta<\hat u_\theta$. Hence, for every $v\in H_0^1(\Omega)$ (with $\tE(v)=+\infty$ allowed), $\tE(v)\ge\tE(\hat v)=E(\hat v)\ge\min E$, where $\min E$ is attained by Lemma~\ref{lemma-E}. Thus $\min\tE=\min E$ is attained, a function $w$ with $\tE(w)=\min\tE$ satisfies $|w|\le u_\theta$ (otherwise $\tE(\hat w)<\tE(w)$, since the minimizers of $\tW$ are exactly $\pm u_\theta$), and therefore $E$ and $\tE$ have exactly the same minimizers over $H_0^1(\Omega)$.
\redend

\section{Nehari manifold for $E(u)$}
{With the modified energy $\tE$ in hand, we now study the structure of critical points of
$E$ using the Nehari manifold. Recall that the first variation (Fr\'{e}chet derivative) of $E$
at $u$ in the direction $v\in H_0^1(\Omega)$ is defined by
\[
\langle \delta E(u), v \rangle := \int_\Omega \left( \kappa \nabla u \cdot \nabla v + W'(u) v \right) dx,
\]
so that $\langle \delta E(u), v\rangle = 0$ for all $v$ is the weak form of the Euler--Lagrange
equation \eqref{E-L}.}

Define the Nehari manifold for the energy $E$ as
\begin{align}\label{def-S}
	{S:= \left\{ u\in H_0^1(\Omega): \|u\|_{L^\infty(\Omega)}<1 \mbox{ and }
	\langle \delta E(u), u\rangle =0 \right\}.}
\end{align}
\new{Note that in this paper we only concentrate on properties of elements in $S$, not other structures $S$ may possess. To be consistent with the literature (cf. e.g. \cite{Diaz-15}), we keep the terminology Nehari manifold.}
{We note that $S$ contains all critical points of $E$
\patch{with $\|u\|_{L^\infty(\Omega)}<1$ (in particular all minimizers, by Lemma~\ref{lemma-E})}, since any critical point $u$ of $E$
satisfies $\langle \delta E(u), v\rangle = 0$ for all $v\in H_0^1(\Omega)$, and in particular for
$v=u$.}
\patch{Computing directly from the definition of $\delta E$ with $v=u$ (no integration by parts is needed, and $\Delta u$ and $\partial_n u$ do not appear),}
\redbegin
\begin{align}\label{S-property1}
	\langle \delta E(u), u\rangle &= \int_\Omega \left(\kappa|\nabla u|^2 + W'(u)u \right) dx \nn\\
	&=\int_\Omega \left( \kappa|\nabla u|^2 - u^2 \right)dx
	+ \frac{\theta}{2}\int_\Omega  u \ln \left(\frac{1+u}{1-u}\right)\,dx.
\end{align}
\redend
Let $\lambda_1$ be the smallest eigenvalue of the negative Laplace operator with Dirichlet boundary
condition. We easily obtain the following (non-sharp) result.

\begin{lemma} \label{lemma-non-sharp}
If $\kappa\geq 1/\lambda_1$, then $S=\{0\}$ and hence $u\equiv 0$ is the only minimizer for $E$
over $H_0^1(\Omega)$.
\end{lemma}
\begin{proof}
For all $u\in H_0^1(\Omega)\setminus\{0\}$, we have $\int_\Omega |\nabla u|^2\,dx \geq \lambda_1
\int_\Omega u^2\,dx$, with equality if and only if $u$ is an eigenfunction corresponding to
$\lambda_1$. Hence, if $\kappa \geq 1/\lambda_1${, then}
\begin{align}
	\int_\Omega ( \kappa |\nabla u|^2 - u^2) \, dx \geq 0{,}
\end{align}
with equality if and only if $\kappa = 1/\lambda_1$ and $u$ is an eigenfunction corresponding to
$\lambda_1$. Since $u\ln\left(\frac{1+u}{1-u} \right) \geq 0$ for all $-1<u<1$, with equality if
and only if $u=0$, we have
\begin{align}
	\int_\Omega u\ln\left(\frac{1+u}{1-u} \right) \,dx \geq 0{,}
\end{align}
with equality if and only if $u\equiv 0$. Combining {these estimates}, if $\kappa\geq 1/\lambda_1$,
then $\langle \delta E(u), u\rangle >0$ unless $u\equiv 0${, i.e.,} $S=\{0\}$.
\end{proof}

For any $u\in H_0^1(\Omega)$ with $|u| \leq u_\theta$, define \new{the fiber}
\begin{align}
	\Phi_u(s) := E(su) \quad\mbox{ for all } s\in \R.
\end{align}
If $u$ is a minimizer for $E$, then $s=1$ is a minimizer for $\Phi_u$. Since $|u|\leq u_\theta<1$,
we need only to consider \new{$-1/u_\theta < s< 1/u_\theta$}. Computing,
\begin{align}
	\Phi'_u(s) & =  \int_\Omega \left\{s(\kappa |\nabla u|^2 - u^2)
	+  \frac{\theta u}{2} \ln \left( \frac{1+s u}{1-s u}\right)\right\} \,dx, \label{Phi'}\\
	\Phi''_u(s) & = \int_\Omega \left( \kappa|\nabla u|^2 - (1-\theta) u^2
	+ \theta s^2 u^4 (1+ s^2u^2 + s^4 u^4 + \cdots)  \right)\,dx. \label{Phi''}
\end{align}
Using the detailed information provided by {\qref{Phi''}}, we can {sharpen}
Lemma~\ref{lemma-non-sharp}.

\begin{lemma}\label{lemma-sharp1}
	If $\kappa \geq (1-\theta)/\lambda_1$, then $u\equiv 0$ is the only minimizer for $E$ over
	$H_0^1(\Omega)$ and
	\[\min \{ E(u): u\in H_0^1(\Omega)\} = \frac{|\Omega|}{2}.\]
\end{lemma}
\begin{proof}
By \qref{Phi'}, we have $\Phi'_u(0) =0$ \new{ and $s=0$ is a critical point of $\Phi_u$.} 
If $\kappa \geq (1-\theta)/\lambda_1$, then
\[\int_\Omega \left( \kappa|\nabla u|^2 - (1-\theta) u^2\right)dx \geq 0\]
for all $u\in H_0^1(\Omega)$. Then for all $u\in H_0^1(\Omega)\setminus \{0\}$ \new{ with $|u|\leq u_\theta$ a.e. in $\Omega$ }
and all \new{$s$ with  $-1/u_\theta<s<1/u_\theta$}, we have $\Phi''_u(s)>0$ \new{ and $\Phi_u$ is strictly convex on $(-1/u_\theta, 1/u_\theta)$.} 
Hence $\Phi_u(s)$ has no other critical points \new{between $-1/u_\theta$ and $1/u_\theta$. In particular, $s=1$ is not a critical point for $\Phi_u$. Consequently $u$ is not a critical point of $E$.} 
Therefore, $u\equiv 0$ is the only critical point of $E$ \new{in the set $\{u\in H_0^1(\Omega): |u|\leq u_\theta \mbox{ a.e. in } \Omega$ \}.}
\redbegin
By Lemma~\ref{lemma-E} a minimizer $u^*$ of $E$ exists and $|u^*|\le u_\theta<1$ a.e. Then $W'(u^*)\in L^\infty(\Omega)$, so $v\mapsto\langle\delta E(u^*),v\rangle$ is a bounded linear functional on $H_0^1(\Omega)$. For $v\in H_0^1(\Omega)\cap L^\infty(\Omega)$ we have $|u^*+tv|<1$ for $|t|$ small, so $t\mapsto E(u^*+tv)$ is differentiable at $t=0$ with derivative $\langle\delta E(u^*),v\rangle$, which vanishes because $u^*$ is a minimizer; by density, $\langle\delta E(u^*),v\rangle=0$ for all $v\in H_0^1(\Omega)$. In particular $\Phi_{u^*}'(1)=\langle\delta E(u^*),u^*\rangle=0$, so $u^*$ is a critical point of $E$ in $\{|u|\le u_\theta\}$, and by the argument above $u^*\equiv0$. Hence $u\equiv0$ is a minimizer, it is the only one, and $\min E=E(0)=|\Omega|/2$.
\redend
\end{proof}

Our next step is to \new{show that there are nontrivial minimizers for $E$ in $H_0^1(\Omega)$ when
$\kappa < (1-\theta)/\lambda_1$. This is an immediate consequence of the following lemma. }
\new{
\begin{lemma}\label{lemma-nontrivial}
    Suppose $0<\kappa < (1-\theta)/\lambda_1.$ Let $\phi$ be an eigenfunction  corresponding to $\lambda_1$ for the negative Laplace 
    operator with homogeneous Dirichlet boundary condition. There exists $s_\phi>0$ such that $\Phi_\phi(s)$ is strictly 
    decreasing for $s\in(0,s_\phi)$ and $E(s_\phi\phi)<E(0)$. Consequently minimizers of $E$ over 
    $H_0^1(\Omega)$ are not constantly zero.
\end{lemma}
}
\begin{proof}
\new{
Since $\phi$ is an eigenfunction  corresponding to $\lambda_1$ for the negative Laplace operator with homogeneous Dirichlet boundary condition, it 
satisfies 
\[ \int_\Omega |\nabla\phi|^2\,dx = \lambda_1 \int_\Omega \phi^2\,dx \] 
and is smooth, and hence bounded. Then there exists $r_1>0$ such that $r_1|\phi|<u_\theta$ for all $x\in\Omega$. 
For $0<s<r_1$, by \qref{Phi'} and the Taylor expansion we have 
\begin{align}
	\Phi'_\phi(s)
	& = s \int_\Omega \left\{ \kappa |\nabla \phi|^2 - (1-\theta) \phi^2
	+ \theta s^2 \phi^4 \left( \frac{1}{3} + \frac{s^2\phi^2}{5} +\cdots\right) \right\} dx \nn\\
    & = s \int_\Omega \left\{ (\kappa\lambda_1 - (1-\theta)) \phi^2
	+ \theta s^2 \phi^4 \left( \frac{1}{3} + \frac{s^2\phi^2}{5} +\cdots\right) \right\} dx \nn\\
    &\leq s \int_\Omega \left\{ (\kappa\lambda_1 - (1-\theta)) \phi^2
	+   \frac{\theta s^2 \phi^4}{3(1-s^2\phi^2)}  \right\} dx \nn\\
    & = s \int_\Omega \left( \kappa\lambda_1 - (1-\theta)
	+   \frac{\theta s^2 \phi^2}{3(1-s^2\phi^2)} \right)\phi^2  dx.
\end{align}
By the uniform boundedness of $\phi^2$, there exists $r_2>0$ such that 
\[ \kappa\lambda_1 - (1-\theta) +\frac{\theta s^2 \phi^2}{3(1-s^2\phi^2)} <0 \quad \mbox{ for all } 0<s<r_2 \mbox{ and all }x\in\Omega.\]
Hence $\Phi'_\phi(s)<0$ and $\Phi_\phi(s)$ is strictly decreasing for $0<s<r:=\min\{r_1, r_2\}$.
}
\patch{Take $s_\phi:=r/2$. Since $\Phi_\phi$ is continuous on $[0,r)$ and strictly decreasing on $(0,r)$, we have $E(s\phi)<E(0)$ for every $s\in(0,r)$, and in particular}
\[ E(s_\phi \phi) = \Phi_\phi (s_\phi) < \Phi_\phi(0)
= E(0) = \frac{|\Omega|}{2}.\]
\new{Hence $\min\{ E(u): u\in H_0^1(\Omega)\} < E(0)$. In other words, any
minimizer $u^*$ of $E$ over $H_0^1(\Omega)$ cannot be constantly zero.
}
\end{proof}

\section{\new{Minimizers along fibers} for $\tE(u)$}

\new{The argument in Lemma \ref{lemma-nontrivial} can be used to prove that for any $u\in H_0^1(\Omega)\cap L^\infty(\Omega)$ 
satisfying 
\[\int_\Omega \left(\kappa|\nabla u|^2 - (1-\theta)u^2\right)dx <0,\] there exists $s_u>0$ such that $\Phi_u(s)$
is strictly decreasing on $(0, s_u)$. However, we do not have an estimate on how big $s_u$ could 
be since $\Phi_u'(s)$ is potentially only defined for small values of $s$ due to the constraint that $|s u|\leq 1$ a.e. in $\Omega$.
In contrast, our modified energy functional $\tilde E$ defined in \qref{def-tilde_E} does not have such a constraint.}


\new{
For $u\in H_0^1(\Omega)$ \new{satisfying $\tilde E(u)<\infty$}, define the fiber
\[ \tilde{\Phi}_u(s) := \tE(su) \quad\mbox{ for } s\in\R. \] 
Due to the symmetry $\tilde\Phi_{u}(s) = \tilde\Phi_{-u}(-s)$ for all $s\in\R$, we only need to consider the positive half of the 
fiber, i.e., for $s>0$. For any $s>0$ and $h>0$, by \qref{W1'-taylor} and \qref{def-tW1} 
\begin{align}
\tilde W((s+h)u) -\tilde W(su) &= \tilde W_1((s+h)u) - \tilde W_1(su) - \frac12( (s+h)^2-s^2)u^2   \nn\\
&\geq \theta \int_{s|u|}^{(s+h)|u|} \left( t + \frac{1}{3} t^3 + \cdots+ \frac{1}{2k+1}t^{2k+1}\right)dt 
- \frac12 (2sh+h^2)u^2 \nn\\
&\geq hs\theta \left(u^2 + \frac13 s^2u^4+\cdots+ \frac{1}{2k+1}s^{2k}u^{2k+2} \right) - \frac12 (2sh+h^2)u^2. \nn
\end{align}
So
\begin{align}
    \tilde\Phi_u(s+h) - \tilde\Phi_u(s) 
    &= \int_{\Omega} \left\{ \frac\kappa2((s+h)^2-s^2)|\nabla u|^2 + \tilde W((s+h)u) - \tilde W(su)\right\}dx \nn\\
    & =\int_{\Omega} \frac\kappa2(2sh + h^2)|\nabla u|^2\,dx + \int_\Omega (\tilde W((s+h)u) - \tilde W(su))dx,\nn
\end{align}
and the lower right Dini derivative  
\begin{align}
   D_+ \tilde\Phi_u(s) &=\liminf_{h\to 0^+} \frac{ \tilde\Phi_u(s+h) - \tilde\Phi_u(s)}{h} 
    = \int_\Omega s\kappa |\nabla u|^2 + 
     \liminf_{h\to 0^+} \int_\Omega \frac{\tilde W((s+h)u) -\tilde W(su)}{h} \nn\\
     &\geq s\int_\Omega \left\{ \kappa |\nabla u|^2 - (1-\theta) u^2
	+ \theta s^2 u^4 \left( \frac{1}{3} + \frac{(su)^2}{5}
	+\cdots+ \frac{(su)^{2k-2}}{2k+1}\right) \right\} dx. \label{Dini-1}
\end{align}
}

\new{
If $\kappa<(1-\theta)/\lambda_1$, any eigenfunction $\phi$ corresponding to $\lambda_1$ for the negative Laplace operator
with homogeneous Dirichlet boundary condition satisfies $\int_\Omega \left(\kappa|\nabla \phi|^2 - (1-\theta)\phi^2\right)dx <0$. 
Using the same argument as in Lemma \ref{lemma-nontrivial}, we can prove that $\tilde \Phi_\phi(s)$ is decreasing for 
an interval $(0,r)$ where $r>0$. 
Meanwhile, by \qref{Dini-1}, 
\begin{align}
	D_+\tilde{\Phi}_\phi(s) > s\int_\Omega \left(\kappa |\nabla \phi|^2 - (1-\theta)|\phi|^2
	+ \frac{\theta s^2 \phi^4}{3}\right)dx > 0 \quad \mbox{ for all } s > s_\phi^*
\end{align}
where 
\[
s_\phi^*:= \left( - \dfrac{3\int_\Omega \left( \kappa |\nabla \phi|^2
- (1-\theta)\phi^2 \right)dx}{\theta\int_\Omega \phi^4\,dx}\right)^{1/2}.
\]
Letting 
$s_\phi:=\sup\{ r>0: \tilde \Phi_\phi(s) \mbox{ is strictly decreasing on } (0,r)\}$, 
we have $s_\phi\leq s_\phi^*$ and}
\patch{$s_\phi>0$. Here we used that $\phi$ is bounded, so $\tilde\Phi_\phi(s)=\tE(s\phi)$ is finite and continuous for all $s\in\R$ (dominated convergence, since $\tW$ is continuous and $s\phi$ ranges in a bounded set); a continuous function whose lower right Dini derivative is positive on an interval is strictly increasing there, and a function that is strictly decreasing on $(0,r)$ for every $r<s_\phi$ is, by continuity, strictly decreasing on $(0,s_\phi]$.}
\[ \tE(s_\phi \phi) = \tilde{\Phi}_\phi (s_\phi) < \tilde{\Phi}_\phi(0)
= \tE(0) = E(0)= \frac{|\Omega|}{2}.\]
Hence $\min\{ \tE(u): u\in H_0^1(\Omega)\} < \frac{|\Omega|}{2}$. As {noted above}, any
minimizer $u^*$ of $\tE$ must satisfy $u^* \in [-u_\theta, u_\theta]$ for all $x\in \Omega$.
Hence $u^*$ {is also} a minimizer for $E${, since}
$E(u^*) = \tE(u^*)\leq \tE(u) \leq E(u)$ for all $u\in H_0^1(\Omega)$.

\section{Properties of the Minimizers}
Sections 3 and 4 established the existence of minimizers: when $\kappa < (1-\theta)/\lambda_1$,
the energy $E$ has at least one nontrivial minimizer $u^*$, \new{which satisfies 
$|u^*|\leq u_\theta$ for all $x\in\Omega$ by Lemma \ref{lemma-E}}. 
In this section we analyze the
qualitative properties of these minimizers, showing in particular that they come in a symmetric
pair $\pm u^*_+$ and that $0 < u^*_+ < u_\theta$ pointwise in $\Omega$. We further prove 
that there is only one such pair of minimizers.

We define the \new{nonnegative} and \new{nonpositive} minimizers as $u^*_+:=|u^*|$ and $u^*_-=-|u^*|$,
respectively. Clearly $u^*_+, u^*_-\in H_0^1(\Omega)${; moreover,} $W(u^*_+)=W(u^*_-)${,}
and hence $E(u^*_+)=E(u^*_-)$.
\patch{Moreover, $|\nabla|u^*||=|\nabla u^*|$ a.e.\ and $W$ is even, so $E(u^*_\pm)=E(u^*)$, i.e., $u^*_+$ and $u^*_-$ are themselves minimizers.}
By definition, $u^*_+\geq 0${,} $u^*_-\leq 0${,} and
$u^*_++u^*_-=0.$ The Euler--Lagrange equation satisfied by {$u_+^*$} is
\begin{equation}
    \label{E-L}
    -\kappa \Delta {u_+^*}+ W'(u_+^*)=0 \quad {\text{in } \Omega.}
\end{equation}
where 
\[W'(u) = \frac{\theta}{2}\ln\left(\frac{1+{u}}{1-{u}}\right)-{u}.\]
{An analogous equation holds for $u_-^*$, with $u_+^*$ replaced by $u_-^*$.}
The regularity of $u_+^*$ {(and similarly $u_-^*$)} follows from the fact that \new {$|u^*_+|\leq u_\theta < \hat u_\theta$ and 
for each $\theta{\in(0,1)}$, $W'\in C^\infty ([-\hat u_\theta, \hat u_\theta])$.}
Thus, via the Sobolev embeddings and a bootstrap argument, we have 
\new{${u_+^*}\in
C^{\infty}(\Omega)\cap C(\bar{\Omega})$ if $\partial\Omega$ is piecewise $C^1$.}

\subsection{Maximum principle}

The \new{nonegative} minimizer $u_+^*$ satisfies {the boundary value problem}
\begin{equation}
\label{E-L-pos}
\left\{
\begin{aligned}
    -\kappa \Delta u_+^* + \frac{\theta}{2}\ln\left(\frac{1+u_+^*}{1-u_+^*}\right) - u_+^* &= 0
    && \text{in } \Omega, \\
    u_+^* &= 0 && \text{on } \partial\Omega, \\
    u_+^* &\geq 0 && \text{in } \Omega.
\end{aligned}
\right.
\end{equation}

\begin{lemma}\label{lemma-positive} Let $0<\kappa<(1-\theta)/\lambda_1$.
The \new{nonnegative} minimizer $u_+^*$ \new{is a positive minimizer, in the sense that} 
$u_+^*>0$ in $\Omega.$
\end{lemma}
\begin{proof}
    Suppose for contradiction that there exists $x_0 \in \Omega$ such that
    $u_+^*(x_0)=0.$ 
    Since $u_+^*$ {satisfies} \eqref{E-L-pos}{,}
    {a} Taylor expansion of $\frac{\theta}{2}\ln\left(\frac{1+u_+^*}{1-u_+^*}\right)$ around
    $0$ {gives}
    \[-\kappa \Delta u_+^*+\frac{\theta}{3}(u_+^*)^3+\frac{\theta}{5}(u_+^*)^5+\cdots=(1-\theta)u_+^*\geq 0.\]
    {Setting} $c{:=}\frac{\theta}{3}(u_+^*)^2+\frac{\theta}{5}(u_+^*)^4+\cdots{
    {\ge 0}}${,} we
    {obtain}
    \[-\kappa \Delta u_+^*+cu_+^*\geq(1-\theta)u_+^*\geq 0.\]
    \patch{Here $c$ is bounded, since $0\le u_+^*\le u_\theta<1$.}
    By the strong maximum principle, \new{$u_+^*$ cannot attain a nonpositive minimum in the interior of 
    $\Omega$ unless $u_+^*$ is a constant. \patch{Alternatively, since $W'\le0$ on $[0,u_\theta]$, we have $-\kappa\Delta u_+^*=-W'(u_+^*)\ge0$, so $u_+^*$ is superharmonic and the strong minimum principle applies directly.} Since 
    $u_+^*(x_0) = 0 =\min\{ u_+^*(x): x\in\bar\Omega\}$, we have} $u_+^*\equiv 0.$ This contradicts 
    the assumption that $u_+^*$ is a nontrivial minimizer, and hence $u_+^*>0$ {in $\Omega$}. 
\end{proof}

\new{
\begin{lemma}\label{lemma-upper}
    The positive minimizer $u_+^*$ for $E$ satisfies $u_+^* < u_\theta$ in $\Omega.$
\end{lemma}
}
\begin{proof}
\new{Recall that we have $0<u_+^*\leq u_\theta$, and that $W'(u_\theta)=0.$ Then by the mean value theorem, 
\begin{align}
    W'(u_+^*) & = W'(u_+^*) - W'(u_\theta) = W''(\xi)(u_+^* - u_\theta) 
\end{align}
where $u_+^*(x) \leq \xi(x) \leq u_\theta <1$ for all $x\in\Omega$. Now 
\begin{align}
    W''(\xi ) & = (\theta -1) + \theta\xi^2 + \theta\xi^4 + \cdots = \theta-1 + \frac{\theta\xi^2}{1-\xi^2}.
\end{align}
Plugging into \qref{E-L}, we get 
\begin{align}
    -\kappa \Delta u_+^* + \left(\theta-1 + \frac{\theta\xi^2}{1-\xi^2} \right)(u_+^*-u_\theta) = 0.
\end{align}
Let $w := u_\theta - u_+^*$. Then $ w \geq 0 $ in $\Omega$ and satisfies  
\begin{align}
    -\kappa \Delta w + \frac{\theta\xi^2}{1-\xi^2} w = (1-\theta) w \geq 0 \quad\mbox{ in }\Omega.
\end{align}
Since $\frac{\theta\xi^2}{1-\xi^2} \geq 0$, by the strong maximum principle, $w$ cannot attain
a nonpositive minimum in the interior of $\Omega$ unless $w$ is a constant. Since $w$ is not a constant, 
we have $w>0$ in $\Omega$ and hence $u_+^* < u_\theta$ in $\Omega$.
}  
\end{proof}

\patch{An immediate consequence of Lemma \ref{lemma-positive} is that nontrivial 
minimizers of $E$ are either strictly positive in $\Omega$ or strictly negative in $\Omega$. In deed, let $u$ be any nontrivial minimizer. 
Since $E(|u|)=E(u)$, the function $|u|$ is a nonnegative nontrivial minimizer, and $u_+^*:=|u|>0$ in $\Omega$ by Lemma~\ref{lemma-positive}. 
Since $u$ is continuous in $\Omega$ and $\Omega$ is connected, $u$ does not change sign, i.e.\ $u=u_+^*$ or $u=-u_+^*$. 
Next we will show the uniqueness of 
positive and negative minimizers. }

\redbegin
\begin{proposition}\label{prop-unique}
Let $0<\kappa<(1-\theta)/\lambda_1$. Then $E$ has exactly one positive nontrivial minimizer $u_+^*$ in $H_0^1(\Omega)$, and every 
nontrivial minimizer of $E$ in $H_0^1(\Omega)$ equals either $u_+^*$ or $-u_+^*$.
\end{proposition}
\begin{proof}
We only need to show the uniqueness of positive minimizers.  Let $m:=\min \{E(u): u\in H_0^1(\Omega)\}$ 
and let $u_1,u_2\in H_0^1(\Omega)$ be positive minimizers. By Lemma~\ref{lemma-E} and Lemmas~\ref{lemma-positive}, \ref{lemma-upper}, 
$0 < u_i < u_\theta<1$ a.e. in $\Omega$, $i=1,2$. 
Define
\[
w:=\left(\frac{u_1^2+u_2^2}{2}\right)^{1/2}.
\]
The map $(a,b)\mapsto\big((a^2+b^2)/2\big)^{1/2}$ is Lipschitz and vanishes at the origin, so $w\in H_0^1(\Omega)$, and 
$0< w < u_\theta$ a.e. in $\Omega$. Then $\nabla w=\dfrac{u_1\nabla u_1+u_2\nabla u_2}{2w}$. By the Cauchy--Schwarz inequality and $\sqrt{u_1^2+u_2^2}=\sqrt2\,w$,
\[
|\nabla w|\le\frac{\sqrt{u_1^2+u_2^2}\,\sqrt{|\nabla u_1|^2+|\nabla u_2|^2}}{2w}
=\left(\frac{|\nabla u_1|^2+|\nabla u_2|^2}{2}\right)^{1/2}.
\]
Hence
\begin{equation}\label{conv-grad}
\int_\Omega|\nabla w|^2dx\le\frac12\int_\Omega|\nabla u_1|^2dx+\frac12\int_\Omega|\nabla u_2|^2dx.
\end{equation}

For the potential term, let $g(s):=W(\sqrt s)$ for $s\in[0,1)$. By \qref{W1'-taylor}, $W'(t)=\theta\sum_{j\ge0}\frac{t^{2j+1}}{2j+1}-t$, so
\[
g'(s)=\frac{W'(\sqrt s)}{2\sqrt s}=\frac\theta2\sum_{j\ge0}\frac{s^j}{2j+1}-\frac12,
\qquad
g''(s)=\frac\theta2\sum_{j\ge1}\frac{j\,s^{j-1}}{2j+1}>0 ,
\]
so $g$ is strictly convex on $[0,1)$. Since $w^2=\frac12(u_1^2+u_2^2)$ and $u_i^2\in[0,u_\theta^2]\subset[0,1)$,
\begin{equation}\label{conv-pot}
W(w)=g\Big(\frac{u_1^2+u_2^2}{2}\Big)\le\frac12\,g(u_1^2)+\frac12\,g(u_2^2)=\frac12W(u_1)+\frac12W(u_2)
\quad\text{a.e.},
\end{equation}
with strict inequality at every point where $u_1\ne u_2$. All integrals are finite because $W$ is bounded on $[0,u_\theta]$.

From \qref{conv-grad} and the integral of \qref{conv-pot} we have
\[
E(w)\le\tfrac12E(u_1)+\tfrac12E(u_2)=m,
\]
with strict inequality if $u_1\ne u_2$ on a set of positive measure. But $w\in H_0^1(\Omega)$ is admissible, so $E(w)\ge m$. Hence equality holds throughout, which forces $u_1=u_2$ a.e.
\end{proof}
\redend

\section{Numerical simulations}

To carry out numerical simulations, we take the same setting as in \cite{DLL}. We take the domain
to be $\Omega :=(0,L)\times(0,L)$ with $L = \sqrt{{2}}\pi${,} so that $\lambda_1 = 1$.
The approach is to solve the Allen--Cahn equation
\begin{align}
	u_t & = \kappa \Delta u - W'(u), \quad \mbox{in }\Omega,
    \label{AC-eq1} \\
	u & = 0 \quad \mbox{on }\partial\Omega, 
    \label{AC-eq2}
\end{align}
whose equilibrium states are critical points of $E$ under the homogeneous Dirichlet condition. The
numerical scheme {uses} the standard forward Euler {method} in time and {a} five-point
central difference {approximation} for the Laplace operator \new{and reads as follows. For $n\ge 0$,
given $U^n$ satisfying~\eqref{AC-eq2}, find $U^{n+1}$ also satisfying~\eqref{AC-eq2} such that
\begin{align}\label{num-sol}
\frac{U_{i,j}^{n+1} - U_{i,j}^n}{\Delta t}
&=\kappa\Big(\frac{U_{i+1,j}^n + U_{i,j+1}^n - 4 U_{i,j}^n +
U_{i-1,j}^n + U_{i,j-1}^n 
}{h^2}\Big)-W'(U_{i,j}^n), \quad i, \ j=1,\dots, N.
\end{align}
Since the numerical solution $U^{n+1}$ at the boundary nodes is already known, we merely solve~\eqref{num-sol} at the 
interior grid points. As $U^0,$ w}e take random positive initial values, with the {expectation} that the solutions {of} \qref{AC-eq1}{--}\qref{AC-eq2}
will converge to the non-negative global minimizer $u^*_+$. \new{In the scheme~\eqref{num-sol}, $N$ refers to the number of subdivisions along 
each side of $\Omega$ so that} the grid size is $h = L/N$ with $N=128$, and the time step is $\Delta t = 1\times 10^{-4}$. 

\new{Numerically this 
choice of time step guarantees the stability and structure-preserving properties of the simulations. Although a theoretical proof is lacking for this
scheme with homogeneous Dirichlet boundary condition, we 
refer the readers to Theorems 1 and 2 in~\cite{ham24} for the 
maximum principle and stability conditions for this scheme with homogeneous Neumann boundary condition for 
the Allen-Cahn equation with a logarithmic potential. 
We run these simulations to at 
least $ t=50$.} We
use a stopping criterion similar to that in \cite{DLL}. \new{That is,  we introduce the error term $err^n$ defined as
\begin{align}
  err_{i,j}^n  &=\kappa\Big(\frac{U_{i+1,j}^n + U_{i,j+1}^n - 4 U_{i,j}^n +
U_{i-1,j}^n + U_{i,j-1}^n 
}{h^2}\Big)-W'(U_{i,j}^n), \quad i, \ j=1,\dots, N.
\end{align}
We check the $L^\infty$ norm of  $err^n$ at  $t=50, 100, 150, \dots$ and stop at the first time 
for which $\|err^n\|_{\infty} < 1\times 10^{-7}$. 
}
Once this state is reached, we regard it as a
numerical equilibrium and take it as an approximation of $u^*_+$.

Since $\lambda_1 =1$, {a nontrivial positive minimizer for $E$ exists only when
$\kappa < 1-\theta$.} In addition, the positive minimizer $u_\theta$ of the {Flory--Huggins}
potential $W$ {depends} on $\theta$. We use Newton's method to numerically find $u_\theta$ and
compare {it} with the maximum value of the numerical equilibrium of the Allen--Cahn equation
for $\kappa = 0.02$. See Table~\ref{tab1} and Figure~\ref{fig1}. {By
Theorem~\ref{th-main}~(\new{ii}), the maximum value of the positive minimizer $u^*_+$ is strictly less
than $u_\theta$.} {The numerical results in Table~\ref{tab1} are consistent with this
theoretical bound, up to numerical errors.}

\begin{table}[h!]
\begin{center}
\begin{tabular}{cccc}
\toprule
$\theta$ & $u_\theta$ & $\kappa$ & $\max\, u^*_+$ (numerical) \\
\midrule
 0.3  & 0.997414 & 0.02 & 0.997414\\
 0.5  & 0.957504 & 0.02 & 0.957504 \\
 0.7  & 0.828635 & 0.02 & 0.828634 \\
 0.9  & 0.525430 & 0.02 & 0.523093 \\
 0.95 & 0.379485 & 0.02 & 0.356520 \\
\bottomrule
\end{tabular}
\caption{{Comparison of the maximum value of the numerical solution $u^*_+$ with $u_\theta$,
for fixed $\kappa = 0.02$ and varying $\theta$. In all cases, the numerical maximum is strictly
below $u_\theta$, consistent with Theorem~\ref{th-main}~(\new{ii}).}}
\label{tab1}
\end{center}
\end{table}

\begin{figure}[h!]
\begin{center}
\includegraphics[width=2.0in]{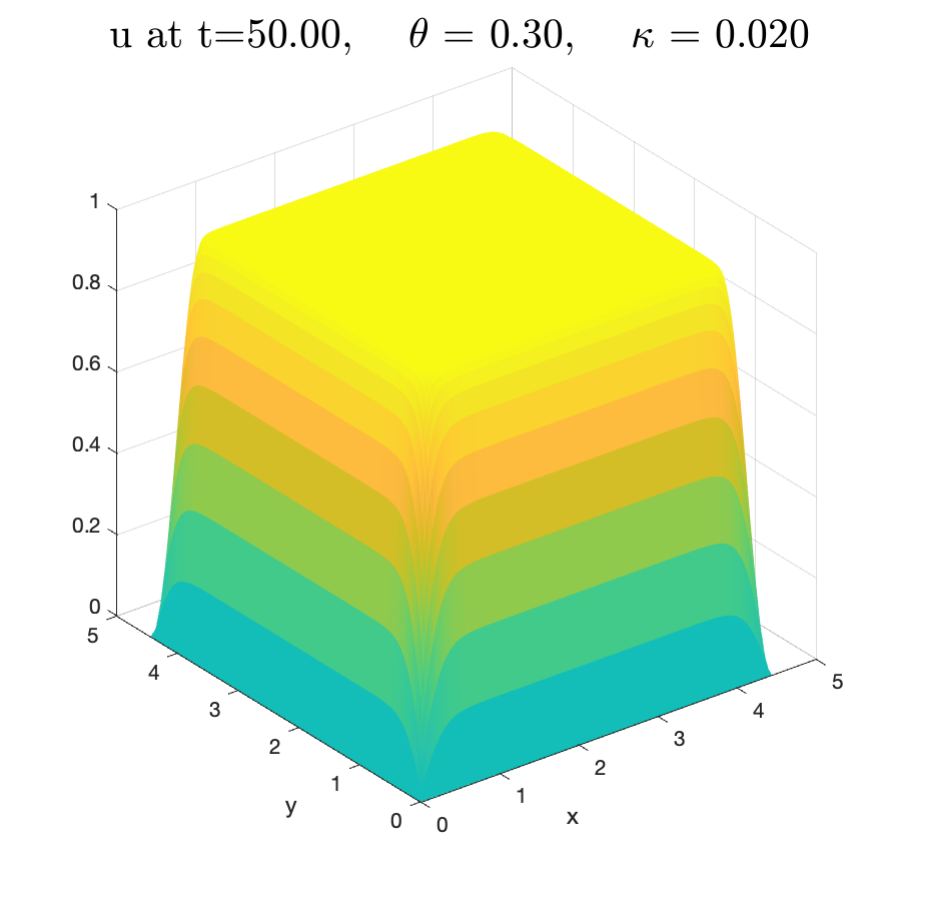}
\includegraphics[width=2.0in]{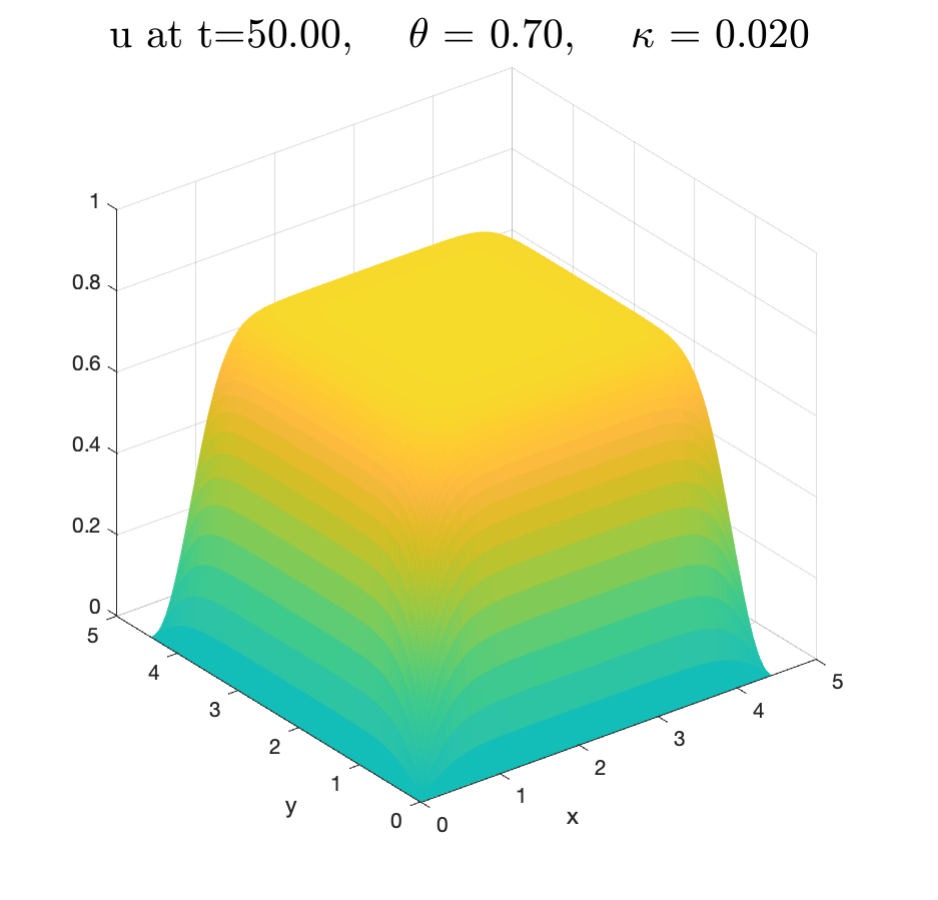}
\includegraphics[width=2.0in]{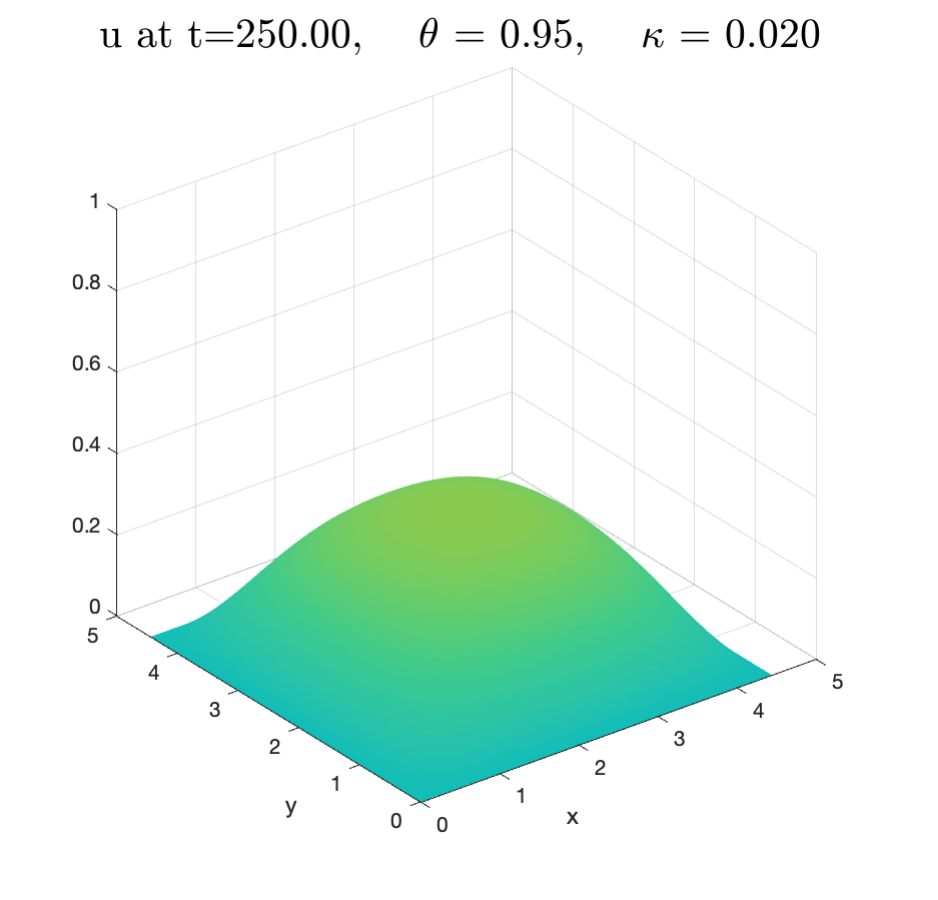}
\end{center}
\caption{{Positive minimizers $u^*_+$ for the same value of $\kappa$ ($\kappa = 0.02$) but
different values of $\theta$. Note the difference in the time required for the system to reach
equilibrium.} Left: $\theta = 0.3$, middle: $\theta = 0.7$, right: $\theta= 0.95$.}
\label{fig1}
\end{figure}

In \cite{DLL} it was observed that{,} for the quartic potential, {smaller values of
$\kappa$ yield larger maximum values of $u^*_+$.} This {phenomenon} {also} occurs for the
{Flory--Huggins} potential{, as reported here.} {An informal justification is} that
$u^*_+$ should approach the homogeneous state $u\equiv 0$ as $\kappa$ increases {toward}
$(1-\theta)/\lambda_1$. Table~\ref{tab2} and Figure~\ref{fig2} show the numerical results for
$\theta=0.7$.

\begin{table}[h!]
\begin{center}
\begin{tabular}{cccc}
\toprule
$\theta$ & $u_\theta$ & $\kappa$ & $\max\, u^*_+$ (numerical) \\
\midrule
 0.7 & 0.828635 & 0.02  & 0.828634\\
 0.7 & 0.828635 & 0.05  & 0.828409 \\
 0.7 & 0.828635 & 0.10  & 0.821620 \\
 0.7 & 0.828635 & 0.15  & 0.791735 \\
 0.7 & 0.828635 & 0.20  & 0.717498 \\
 0.7 & 0.828635 & 0.25  & 0.560631 \\
 0.7 & 0.828635 & 0.28  & 0.375849 \\
 0.7 & 0.828635 & 0.299 & 0.087817 \\
\bottomrule
\end{tabular}
\caption{{Positive minimizers $u^*_+$ for fixed $\theta = 0.7$ and varying $\kappa$. Smaller
values of $\kappa$ yield larger maximum values of $u^*_+$, consistent with the \new{expected near-threshold behavior }
 that $u^*_+ \to 0$ as $\kappa \to (1-\theta)/\lambda_1 = 0.3$.}}
\label{tab2}
\end{center}
\end{table}

\begin{figure}[h!]
\begin{center}
\includegraphics[width=2.0in]{_Result_surf_L=4.4429_N=128_theta=0.7000_kappa=0.0200_t=50.0000_run_1.pdf}
\includegraphics[width=2.0in]{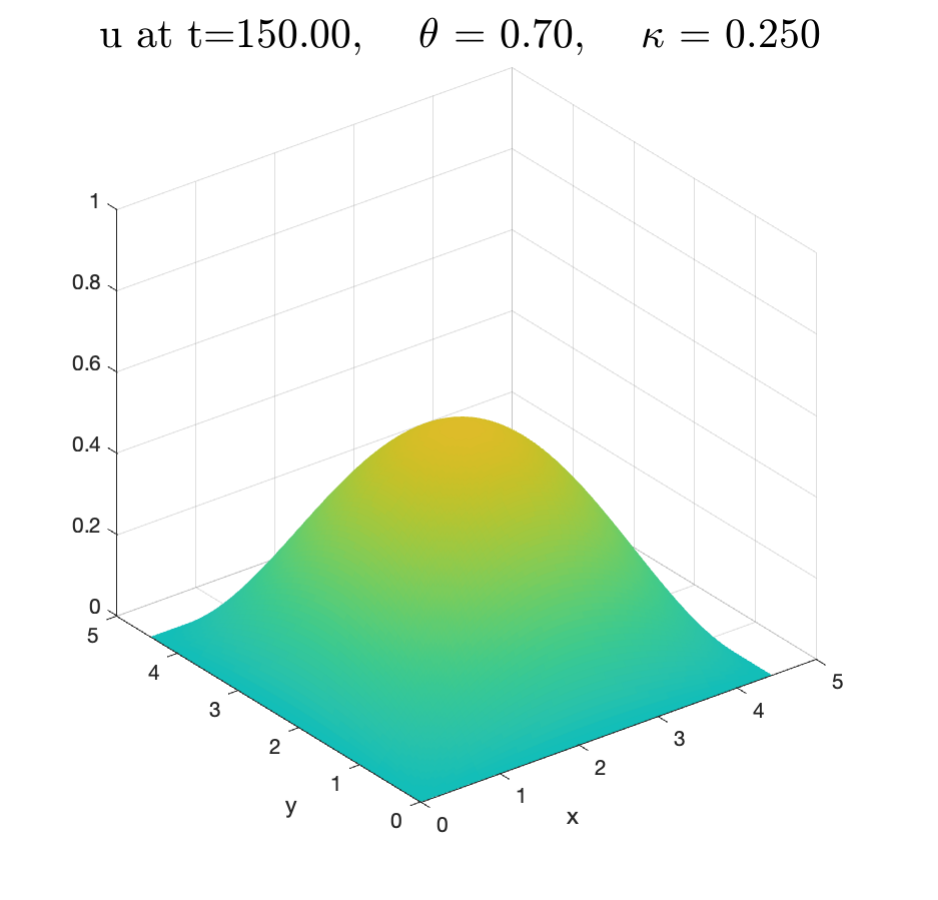}
\includegraphics[width=2.0in]{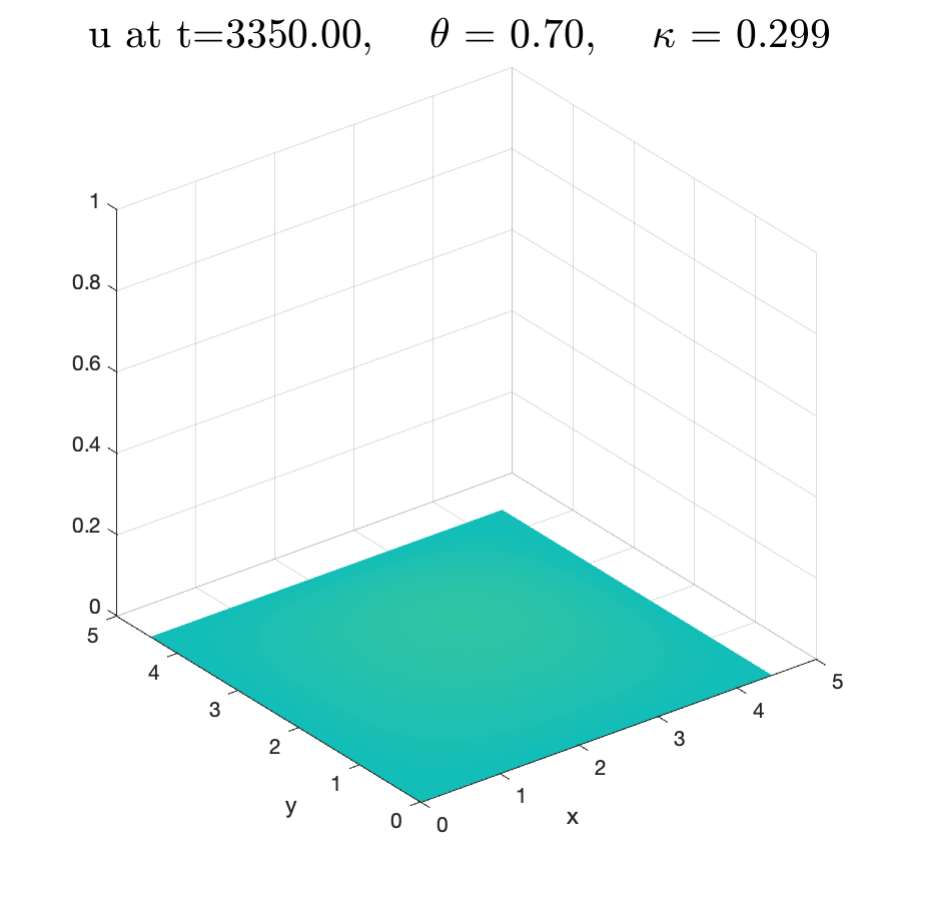}
\end{center}
\caption{{Positive minimizers $u^*_+$ for fixed $\theta = 0.7$ and varying $\kappa$. Note the
difference in the time required for the system to reach equilibrium.} Left: $\kappa = 0.02$,
middle: $\kappa = 0.25$, right: $\kappa= 0.299$.}
\label{fig2}
\end{figure}

\section{Discussion}

This work investigates the minimizers of the Cahn--Hilliard energy functional with the
{Flory--Huggins} potential, often referred to as the logarithmic potential. The boundary
condition considered here is the Dirichlet boundary condition. The energy functional depends on two
key parameters: $\kappa$, which can be interpreted as the {transition layer thickness
parameter}, and $\theta$, which controls the {temperature and hence the} strength of the
nonlinearity.

Our analysis reveals a bifurcation phenomenon for the minimizers when the prescribed boundary value
{is a homogeneous mixture of the two phases}. Depending on the parameter regime, the minimizer
is either unique or there exist two distinct minimizers. The onset of bifurcation depends on the
interplay between {$\kappa$, $\theta$, and} the smallest eigenvalue of the negative Laplacian.

Similar analyses of minimizers have been carried out in previous work. For instance, Dai, Li, and
Luong \cite{DLL} studied the Cahn--Hilliard energy with a quartic potential under the same
anchoring conditions. In that setting, the nonlinearity was smoother and technically easier to
handle. More recently, Dai and Ramadan \cite{Dai-Ramadan:minimizers} considered the
{de~Gennes--Cahn--Hilliard} energy, where the challenge arose from the singular nature of the
potential in combination with dispersion{; the authors overcame this difficulty through a
suitable nonlinear transformation}.

In the present article, {the logarithmic nonlinearity introduces singular behavior that
requires substantially different techniques. At the same time,} the temperature parameter $\theta$
plays a crucial role{: depending on its value, the qualitative properties of minimizers vary,
and $\theta$ enters as an additional bifurcation parameter not present in the quartic or
trigonometric settings.} Furthermore, the methods developed here can be extended to other boundary
conditions, such as Neumann or dynamic boundary conditions, or {mixed} boundary conditions
where different parts of the boundary satisfy different constraints. Such generalizations may lead
to new and potentially richer properties of the minimizers.

Our theoretical analysis is supported by the numerical simulations presented {in Section~6}.
Specifically, we benchmark against the example {in} \cite{DLL}{,} with the quartic
potential replaced {by} the {Flory--Huggins} potential. Our numerical findings are
consistent with the observations in \cite{DLL}. {The choice of the Flory--Huggins potential
introduces the additional temperature parameter $\theta$, and we report on the \new{effects} of varying
both the transition layer thickness and the temperature on the minimizers.} The {singular and
nonlinear} nature of the {logarithmic} potential poses {interesting} numerical
challenge{s}{, particularly in the design of robust numerical schemes.} Coupled with the
incorporation of {random initial data} and the need {for} long{-}time simulations, a careful numerical analysis of this problem remains an interesting open problem{.}

\section*{Conflict of Interest}
The authors declare that they have no conflict of interest.

\section*{Funding}
The work of Shibin Dai was partially supported by the U.S. National Science Foundation through
grants CBET-2212116 and DMS-2345500. The work of Natasha Sharma was supported by the U.S.
National Science Foundation through grant DMS-2110774.

\section*{Data Availability}
The data that support the findings of this study are available from the corresponding author upon
reasonable request.

\bibliographystyle{plain}
\bibliography{mybiblio}

\end{document}